\documentclass{article}
\usepackage{latexsym,amsfonts,amsthm,amsmath,amscd,amssymb}
\usepackage[dvips]{graphicx}

\newtheorem{lemma}{Lemma}[section]
\newtheorem{thm}[lemma]{Theorem}

\newtheorem{prop}[lemma]{Proposition}
\newtheorem{cor}[lemma]{Corollary}
\newtheorem{claim}[lemma]{Claim}

\newtheorem{defn}[lemma]{Definition}

\newcommand\matZ{{\mathbb{Z}}}
\newcommand\matR{{\mathbb{R}}}
\newcommand\matN{{\mathbb{N}}}

\renewcommand{\hbar}{{\overline{h}}}

\newfont{\Got}{eufm10 scaled 1200}

\newcommand{\mycap} [1] {\caption{\footnotesize{#1}}}

\newcommand{\stab}{\mathop{\rm Stab}\nolimits}
\newcommand{\rist}{\mathop{\rm rist}\nolimits}

\newcommand{\Aut}{\mathop{\rm Aut}\nolimits}
\newcommand{\Sym}{\mathop{\rm Sym}\nolimits}

\newcommand\calD{{\mathcal D}}
\newcommand\calP{{\mathcal P}}

\begin{document}

\title{Groups of tree automorphisms as diffeological groups}

\author{Ekaterina~{\textsc Pervova}}

\maketitle

\begin{abstract}
\noindent We consider certain groups of tree automorphisms as
so-called diffeological groups. The notion of diffeology, due to
Souriau, allows to endow non-manifold topological spaces, such as
regular trees that we look at, with a kind of a differentiable
structure that in many ways is close to that of a smooth manifold; a
suitable notion of a diffeological group follows. We first study the
question of what kind of a diffeological structure is the most
natural to put on a regular tree in a way that the underlying
topology be the standard one of the tree. We then proceed to
consider the group of all automorphisms of the tree as a
diffeological space, with respect to the functional diffeology,
showing that this diffeology is actually the discrete one.

\end{abstract}

\section*{Introduction}

The notion of a \emph{diffeological structure}, or simply
\emph{diffeology}, due to J.M. Souriau \cite{So1,So2}, appeared in
Differential Geometry as part of the quest to generalize the notion
of a smooth manifold in a way that would yield a category closed
under the main topological constructions yet carrying sufficient
geometric information. To be more precise, it is well-known that the
category of smooth manifolds, while being the main object of study
in Differential Geometry, is not closed under some of the basic
topological constructions, such as taking quotients or function
spaces, nor does it include objects which in recent years attracted
much attention both from geometers and mathematical physicists, such
as irrational tori, orbifolds, spaces of connections on principal
bundles in Yang-Mills field theory, to name just a few. Many
fruitful attempts, some of which are summarized in \cite{St}, had
been made to address these issues, notably in the realm of
functional analysis and noncommutative geometry, via smooth
structures \`a la Sikorski or \`a la Fr\"olicher; each of these
attempts however had its own limitations, be that the sometimes
exaggerated technical complexity or missing certain topological
situations (such as singular quotients, missing from Sikorski's and
Fr\"olicher's spaces).

The diffeology, whose birth story is beautifully described in the
Preface and Afterword of the excellent book \cite{iglesiasBook}, has
the advantage of being possibly the least technical (and therefore
very easy to work with) and, much more importantly, very wide in
scope. Indeed, the category of diffeological spaces contains, on one
hand, smooth manifolds as a full subcategory, and is very
well-behaved on the other: in particular, it is complete, cocomplete
and cartesian closed (see, for example, Theorems 2.5 and 2.6 in
\cite{CSW_Dtopology}).

As for diffeological groups, they were in fact the context in which
the notion of diffeology was introduced; the very titles of the
already mentioned foundational papers by Souriau are witnesses to
this fact. More precisely, the historical origin of the concept of
``diffeology'' was, as evidenced by Iglesias-Zemmour's fascinating
account of those events in \cite{iglesiasBook}, Souriau's attempt to
regard some types of coadjoint orbits of infinite dimensional groups
of diffeomorphisms as Lie groups, and to do so in ``the simplest
possible manner''. On the other hand, as mentioned in Chapter 7 of
\cite{iglesiasBook}, the theory of diffeological groups has not yet
been much developed.

\paragraph{What does this have to do with groups of tree
automorphisms?} Before answering this question, it should be useful
to say right away what we mean by a ``tree''; and to give the idea
of what is done in this paper, it should suffice to point out that
all trees under consideration are infinite, rooted, and
\emph{regular}. The meaning of the latter is this: we fix an integer
$p\geqslant 2$ and consider an infinite tree with precisely one
vertex of valence $p$ (this is the root) and all other vertices of
valence $p+1$. Such an object is a very natural venue for applying
the notion of diffeology: on one hand, it is a topological space
quite different from a (one-dimensional) manifold, since it contains
an infinite (albeit discrete) set of points whose local
neighbourhood is a cone over at least three points, and on the other
hand, there is a natural diffeological structure to put on it, the
so-called ``wire diffeology'' (see below). This fact in itself
raises a number of questions, for reasons of intellectual curiosity
at least if nothing else, such as, will the \emph{D-topology} be
different or equal to the standard topology of the tree?

Now, groups of automorphisms of such a tree, even restricted to a
rather specific construction such as the one we will deal with
(which is however independently interesting from the algebraic point
of view, see the foundational paper \cite{GrigFoundation}) are
easily seen to be groups of diffeomorphisms of the tree with respect
to the above diffeological structure. The category of diffeological
spaces being closed under taking groups of diffeomorphisms, they
become in the end diffeological groups; and since they are also
topological groups with respect to, for instance, profinite topology
(but occasionally there are some others, see, for instance,
\cite{mioJalg}), the same questions about comparing the two
topologies arise... And going further still, the question becomes,
\emph{what kind of information about these groups can we obtain if
we regard them as diffeological groups?}

\paragraph{The content} The first two sections are
devoted to recalling some of the main definitions and constructions
related to, respectively, diffeological spaces and (certain kind of)
groups of automorphisms of regular rooted trees; they gather
together everything that is used henceforth, \emph{i.e.} in Sections
3 and 4. The first of these two deals with the choice of the
diffeology to put on the tree, showing in the end that the topology
corresponding to the final choice (the so-called D-topology) is
indeed the one coinciding with that of the tree in the usual sense.
The last section is devoted to the functional diffeology on the
whole group of tree automorphisms, showing that (for reasons that
apply actually to any subgroup of this group) the functional
diffeology is the discrete one; a finding that is not surprising in
view of the discrete nature of these groups that had originated as
so-called \emph{automata groups} \cite{alyoshin}.

\paragraph{Acknowledgments} This was one of the first papers for me
on the subject, and just completing it felt like a minor
accomplishment. For a, maybe indirect, but no less significant for
that, assistance in that moment I must thank Prof. Riccardo Zucchi,
despite his habit of refuting his merits.\footnote{Originally this
acknowledgment included other people; they are not here anymore.
People disappoint and get disappointed, whoever is at fault; I guess
that's life.}

\section{Diffeological spaces}\label{defn-diffeology-sect}

This section is devoted to a short background on diffeological
spaces, introducing the concepts that we will need in what follows.

\paragraph{The concept} We start by giving the basic definition of a
diffeological space, following it with the definition of the
\emph{standard diffeology} on a smooth manifold; it is this latter
diffeology that allows for the natural inclusion of smooth manifolds
in the framework of diffeological spaces.

\begin{defn} \emph{(\cite{So2})} A \textbf{diffeological space} is a pair
$(X,\calD_X)$ where $X$ is a set and $\calD_X$ is a specified
collection of maps $U\to X$ (called \textbf{plots}) for each open
set $U$ in $\matR^n$ and for each $n\in\matN$, such that for all
open subsets $U\subseteq\matR^n$ and $V\subseteq\matR^m$ the
following three conditions are satisfied:
\begin{enumerate}
  \item (The covering condition) Every constant map $U\to X$ is a
  plot;
  \item (The smooth compatibility condition) If $U\to X$ is a plot
  and $V\to U$ is a smooth map (in the usual sense) then the
  composition $V\to U\to X$ is also a plot;
  \item (The sheaf condition) If $U=\cup_iU_i$ is an open cover and
  $U\to X$ is a set map such that each restriction $U_i\to X$ is a
  plot then the entire map $U\to X$ is a plot as well.
\end{enumerate}
\end{defn}

Typically, we will simply write $X$ to denote the pair
$(X,\calD_X)$. Such $X$'s are the objects of the category of
diffeological spaces; naturally, we shall define next the arrows of
the category, that is, say which maps are considered to be smooth in
the diffeological sense. The following definition says just that.

\begin{defn} \emph{(\cite{So2})} Let $X$ and $Y$ be two
diffeological spaces, and let $f:X\to Y$ be a set map. We say that
$f$ is \textbf{smooth} if for every plot $p:U\to X$ of $X$ the
composition $f\circ p$ is a plot of $Y$.
\end{defn}

As is natural, we will call an isomorphism in the category of
diffeological spaces a \textbf{diffeomorphism}. The typical notation
$C^{\infty}(X,Y)$ will be used to denote the set of all smooth maps
from $X$ to $Y$.

\paragraph{The standard diffeology on a smooth manifold} Every
smooth manifold $M$ can be canonically considered a diffeological
space with the same underlying set, if we take as plots all maps
$U\to M$ that are smooth in the usual sense. With this diffeology,
the smooth (in the usual sense) maps between manifolds coincide with
the maps smooth in the diffeological sense. This yields the
following result (see Section 4.3 of \cite{iglesiasBook}).

\begin{thm}
There is a fully faithful functor from the category of smooth
manifolds to the category of diffeological spaces.
\end{thm}

\paragraph{Comparing diffeologies} Given a set $X$, the set of all
possibile diffeologies on $X$ is partially ordered by inclusion
(with respect to which it forms a complete lattice). More precisely,
a diffeology $\calD$ on $X$ is said to be \textbf{finer} than
another diffeology $\calD'$ if $\calD\subset\calD'$ (whereas
$\calD'$ is said to be \textbf{coarser} than $\calD$). Among all
diffeologies, there is the finest one, which turns out to be the
natural \textbf{discrete diffeology} and which consists of all
locally constant maps $U\to X$; and there is also the coarsest one,
which consists of \emph{all} possible maps $U\to X$, for all
$U\subseteq\matR^n$ and for all $n\in\matN$. It is called \emph{the}
\textbf{coarse diffeology} (or \textbf{indiscrete diffeology} by
some authors).

\paragraph{Generated diffeology and quotient diffeology} One notion
that will be crucial for us is the notion of a so-called
\textbf{generated diffeology}. Specifically, given a set of maps
$A=\{U_i\to X\}_{i\in I}$, the \textbf{diffeology generated by $A$}
is the smallest, with respect to inclusion, diffeology on $X$ that
contains $A$. It consists of all maps $f:V\to X$ such that there
exists an open cover $\{V_j\}$ of $V$ such that $f$ restricted to
each $V_j$ factors through some element $U_i\to X$ in $A$ via a
smooth map $V_j\to U_i$. Note that the standard diffeology on a
smooth manifold is generated by any smooth atlas on the manifold,
and that for any diffeological space $X$, its diffeology $\calD_X$
is generated by $\cup_{n\in\matN}C^{\infty}(\matR^n,X)$.

Note that one useful property of diffeology as concept is that the
category of diffeological spaces is closed under taking quotients.
To be more precise, let $X$ be a diffeological space, let $\cong$ be
an equivalence relation on $X$, and let $\pi:X\to Y:=X/\cong$ be the
quotient map. The \textbf{quotient diffeology} (\cite{iglOrb}) on
$Y$ is the diffeology in which $p:U\to Y$ is the diffeology in which
$p:U\to Y$ is a plot if and only if each point in $U$ has a
neighbourhood $V\subset U$ and a plot $\tilde{p}:V\to X$ such that
$p|_{V}=\pi\circ\tilde{p}$.

\paragraph{Sub-diffeology and inductions} Let $X$ be a diffeological
space, and let $Y\subseteq X$ be its subset. The
\textbf{sub-diffeology} on $Y$ is the coarsest diffeology on $Y$
making the inclusion map $Y\hookrightarrow X$ smooth. It consists of
all maps $U\to Y$ such that $U\to Y\hookrightarrow X$ is a plot of
$X$. This definition allows also to introduce the following useful
term: for two diffeological spaces $X,X'$ a smooth map $f:X'\to X$
is called an \textbf{induction} if it induces a diffeomorphism
$X\to\mbox{Im}(f)$, where $\mbox{Im}(f)$ has the sub-diffeology of
$X$.

\paragraph{Sums of diffeological spaces} Let $\{X_i\}_{i\in I}$ be a
collection of diffeological spaces, with $I$ being some set of
indices. The \textbf{sum}, or the \textbf{disjoint union}, of
$\{X_i\}_{i\in I}$ is defined as
$$X=\coprod_{i\in I}X_i=\{(i,x)\,|\,i\in I\mbox{ and }x\in X_i\}.$$
The \textbf{sum diffeology} on $X$ is the finest diffeology such
that the natural injections $X_i\to\coprod_{i\in I}X_i$ are smooth
for each $i\in I$. The plots of this diffeology are maps
$U\to\coprod_{i\in I}X_i$ that are \emph{locally} plots of one of
the components of the sum.

\paragraph{The diffeological product} Let, again, $\{X_i\}_{i\in I}$ be a
collection of diffeological spaces, and let $\calD_i$, $i\in I$, be
their respective diffeologies. The \textbf{the product diffeology}
$\calD$ on the product $X=\prod_{i\in I}X_i$ is the \emph{coarsest}
diffeology such that for each index $i\in I$ the natural projection
$\pi_i:\prod_{i\in I}X_i\to X_i$ is smooth.

\paragraph{Functional diffeology} Let $X$, $Y$ be two diffeological
spaces, and let $C^{\infty}(X,Y)$ be the set of smooth maps from $X$
to $Y$. Let \textsc{ev} be the \emph{evaluation map}, defined by
$$\mbox{\textsc{ev}}:C^{\infty}(X,Y)\times X\to Y\mbox{ and }\mbox{\textsc{ev}}(f,x)=f(x). $$
The words ``functional diffeology'' stand for \emph{any} diffeology
on $C^{\infty}(X,Y)$ such that the evaluation map is smooth; note,
for example, that the discrete diffeology is a functional
diffeology. However, they are typically used, and we also will do
that from now on, to denote \emph{the coarsest} functional
diffeology.

There is a useful criterion for a given map to be a plot with
respect for the functional diffeology on a given $C^{\infty}(X,Y)$,
which is as follows.

\begin{prop}\label{criterio-funct-diff} \emph{(\cite{iglesiasBook}, 1.57)}
Let $X$, $Y$ be two diffeological spaces, and let $U$ be a domain of
some $\matR^n$. A map $p:U\to C^{\infty}(X,Y)$ is a plot for the
functional diffeology of $C^{\infty}(X,Y)$ if and only if the
induced map $U\times X\to Y$ acting by $(u,x)\mapsto p(u)(x)$ is
smooth.
\end{prop}

\paragraph{Diffeological groups} A \textbf{diffeological group} is a
group $G$ equipped with a compatible diffeology, that is, such that
the multiplication and the inversion are smooth:
$$[(g,g')\mapsto gg']\in C^{\infty}(G\times G,G)\mbox{ and }[g\mapsto g^{-1}]\in C^{\infty}(G,G).$$
Thus, it mimicks the usual notions of a topological group and a Lie
group: it is both a group and a diffeological space such that the
group operations are maps (arrows) in the category of diffeological
spaces.

\paragraph{Functional diffeology on diffeomorphisms} Groups of
diffeomorphisms of diffeological spaces being the main examples
known of diffeological groups, and being precisely the kind of
object which we study below, we shall comment on their functional
diffeology. Let $X$ be a diffeological space, and let
$\mbox{Diff}(X)$ be the group of diffeomorphisms of $X$. As
described in the previous paragraph, $\mbox{Diff}(X)$, as well as
any of its subgroups, inherits the functional diffeology of
$C^{\infty}(X,X)$. On the other hand, there is the standard
diffeological group structure on $\mbox{Diff}(X)$ (or its subgroup),
which is the coarsest \emph{group} diffeology such that the
evaluation map is smooth. Note that, as observed in Section 1.61 of
\cite{iglesiasBook}, this diffeological group structure is in
general finer than the functional diffeology (therefore making a
comparison between the two will be part of our task in what
follows).

\paragraph{The D-topology} There is a ``canonical'' topology underlying each
diffeological structure; it is defined as follows:

\begin{defn} \emph{(\cite{iglesiasBook})} Given a diffeological
space $X$, the final topology induced by its plots, where each
domain is equipped with the standard topology, is called the
\textbf{D-topology} on $X$.
\end{defn}

To be more explicit, if $(X,\calD_X)$ is a diffeological space then
a subset $A$ of $X$ is open in the D-topology of $X$ if and only if
$p^{-1}(A)$ is open for each $p\in\calD_X$; we call such subsets
\textbf{D-open}. Note that if $\calD_X$ is generated by some
$\calD'$ then $A$ is D-open if and only if $p^{-1}(A)$ is open for
each $p\in\calD'$.

A smooth map $X\to X'$ is continuous if $X$ and $X'$ are equipped
with D-topology (hence there is an associated functor from the
category of diffeological spaces to the category of topological
spaces). As an important example, it is easy to see that the
D-topology on a smooth manifold with the standard diffeology
coincides with the usual topology on the manifold; in fact, this is
frequently the case even for non-standard diffeologies. That is due
to the fact that, as established in \cite{CSW_Dtopology}, the
D-topology is completely determined smooth curves. More precisely,
the following statement was proven in \cite{CSW_Dtopology}:

\begin{thm} \emph{(Theorem 3.7 of \cite{CSW_Dtopology})} The
D-topology on a diffeological space $X$ is determined by
$C^{\infty}(\matR,X)$, in the sense that a subset $A$ of $X$ is
D-open if and only if $p^{-1}(A)$ is open for every $p\in
C^{\infty}(\matR,X)$.
\end{thm}

\section{Regular trees and subgroups of $\mbox{Aut
T}$}\label{defn-trees-sect}

As already mentioned, we will consider regular rooted trees of
valence $p$; this implies that there is a \emph{root}, of valence
$p$, and all the other vertices have valence $p+1$; such a tree is
naturaly decomposed into \emph{levels}, sets of vertices of equal
distance from the root (this distance being an integer equal to the
number of edges in the shortest path connecting the root to the
vertex in question). Below we give precise definitions of these
concepts and others that we will need.

\paragraph{Regular rooted trees} A regular 1-rooted tree, the simplest
example of which is shown in Fig. \ref{tree:fig}, is naturally
identified with the set of all words in a given finite alphabet $A$
of appropriate cardinality $p$.
\begin{figure}
    \begin{center}
    \includegraphics[scale=0.47]{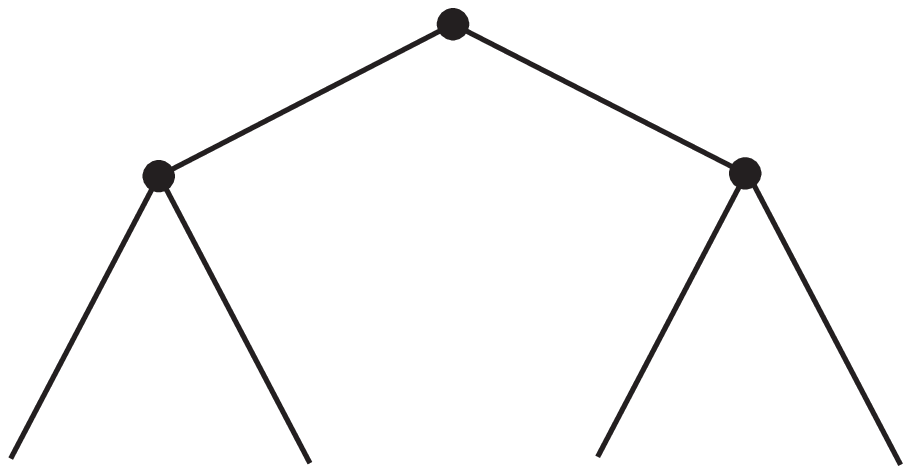}
\vspace{1cm}
    \mycap{An example of a regular 1-rooted tree; in this case $p=2$. The
    figure shows the root and the two vertices of the first level, with all
    the edges incident to them.}
    \label{tree:fig}
    \end{center}
    \end{figure}
Under this identification, the words correspond to vertices, the
root is the empty word, and two vertices are joined by an edge if
and only if they have the form $a_1a_2\ldots a_n$ and $a_1a_2\ldots
a_na_{n+1}$ for some $n$ and some $a_i\in A$. The number $n$ is
called the $\textbf{length}$ of a vertex $u=a_1a_2\ldots a_n$ and is
denoted by $|u|$. The set of all vertices of length $n$ is called
the $n$th \textbf{level} of $T$.

Suppose that $u=\hat{a}_1\hat{a}_2\ldots\hat{a}_n$ is a vertex. The
set of all vertices of the form
$$\hat{a}_1\hat{a}_2\ldots\hat{a}_na_{n+1}a_{n+2}\ldots a_{n+m},$$
where $m\in\matN$ and $a_{n+i}$ range over the set $A$, forms a
subtree of $T$; we will denote this subtree by $T_u$. It is easy to
see that $T_u$ is naturally isomorphic to the same tree $T$ via the
map
$$\hat{a}_1\hat{a}_2\ldots\hat{a}_na_{n+1}a_{n+2}\ldots a_{n+m}\mapsto a_{n+1}a_{n+2}\ldots a_{n+m}.$$
This map allows to identify subtrees $T_u$ for all vertices $u$,
with one fixed tree $T$.

\paragraph{Their automorphism groups} Let $T$ be a tree as above; an
\textbf{automorphism} of $T$ is a bijective map $f$ which fixes the
root and preserves the adjacency of vertices. The set of all
possible automorphisms of $T$ is obviously a group which we denote
by $\Aut T$; note that it is a profinite group\footnote{More
precisely, it is a pro-$p$-group.} (see also below).

\paragraph{Vertex stabilizers and congruence subgroups} Consider now
an arbitrary subgroup $G$ of $\Aut T$ and a vertex $v$ of $T$. The
\textbf{stabilizer} of $v$ in $G$ is the subgroup
$$\stab_G(v)=\{g\in G\,|\,v^g=v\}.$$
Now, if we consider the set of all vertices of level $n$, the
subgroup $\cap_{|v|=n}\stab_G(v)$ is called the ($n$th)
\textbf{level stabilizer} and is denoted by $\stab_G(n)$.

The subgroups $\stab_G(n)$ are also called \textbf{principal
congruence subgroups} in $G$. A subgroup of $G$ which contains a
principal congruence subgroup is in turn called a \textbf{congruence
subgroup}.

\paragraph{Rigid stabilizers} Let once again $G\leqslant\Aut T$ and
$v\in T$ a vertex. The \textbf{rigid stabilizer} of $v$ in $G$ is
the subgroup
$$\rist_G(v)=\{g\in G\,|\,u^g=u\mbox{ for all }u\in T\setminus T_v\}.$$
We also denote by $\rist_G(n)$ the subgroup
$\prod_{|v|=n}\rist_G(v)$; note that this is a normal subgroup of
$G$ (unlike the rigid stabilizer of just one vertex).

\paragraph{Recursive presentation of the action of $\Aut T$} It is
easy to see that $\Aut T$ possesses a sort of ``recurrent''
structure, that we now describe, as it is extremely useful for
working with $\Aut T$ (and its subgroups). Observe that $\Aut T$
admits a natural map $\phi:\Aut T\to\Aut T\wr\Sym(A)$, where
$\Sym(A)$ is the group of all permutations of elements of $A$. Thus,
every element $x$ of $\Aut T$ is given by an element
$$f_x\in\underbrace{\Aut T\times\ldots\times\Aut T}_{|A|}$$
and a permutation $\pi_x\in\Sym(A)$. The latter permutation is
called the \textbf{accompanying permutation}, or the
\textbf{activity}, of $x$ at the root. We write that
$$\phi(x)=f_x\cdot\pi_x.$$

In particular, the restriction of $\phi$ onto $\stab_{\Aut T}(1)$ is
an embedding (actually, an isomorphism) of $\stab_{\Aut T}(1)$ into
(with) the direct product of $|A|$ copies of $\Aut T$; we will
denote this restriction by $\Phi_1$. Furthermore, it is easy to see
that
$$\Phi_1(\stab_{\Aut T}(2))=\underbrace{\stab_{\Aut T}(1)\times\ldots\times\stab_{\Aut T}(1)}_{|A|};$$
therefore we can obtain the isomorphism
$$\Phi_2=(\underbrace{\Phi_1\times\ldots\times\Phi_1}_{|A|})\circ\Phi_1:\stab_{\Aut T}(2)\to
\underbrace{\Aut T\times\ldots\Aut T}_{|A|^2}.$$ Proceeding in this
manner, we define for each positive integer $n$ the isomorphism
$$\Phi_n=(\underbrace{\Phi_{n-1}\times\ldots\times\Phi_{n-1}}_{|A|})\circ\Phi_1:\stab_{\Aut T}(n)\to
\underbrace{\Aut T\times\ldots\Aut T}_{|A|^n}.$$

\paragraph{Profinite topology and congruence topology} Let $G\leqslant\Aut
T$; the \textbf{profinite topology} on $G$ is the topology generated
by all its finite-index subgroups taken as the system of
neighbourhoods of unity. To define the \textbf{congruence topology},
we take the set of all principal congruence subgroups (\emph{i.e.},
the level stabilizers) as the system of neighbourhoods of unity.
These two topologies frequently coincide (as it happens for the
first of the examples described below) but sometimes they do not (as
is the case for the second of the examples that follow).

\paragraph{Examples} Simple examples of the groups described above
can be found in \cite{GrigFoundation} or \cite{mioJalg} (a different
sort). We do not describe them, since we will not need them.

\section{A regular tree as a diffeological space}\label{diffeology-T-sect}

In this section we endow each regular tree $T$ with a diffeology.
The condition imperative in making the choice of such is that the
corresponding D-topology coincide with the usual one.

\subsection{General considerations}

A regular rooted tree, such as the ones we are considering, is not
naturally a smooth object, and a choice of diffeology with which to
endow it, represents its own issue. Although there exist other
options, the one we prefer is a certain analogue of the so-called
\emph{wire diffeology}. The latter was introduced by J.M. Souriau as
a diffeology on $\matR^n$ alternative to the standard one; it is the
diffeology generated by the set $C^{\infty}(\matR,\matR^n)$, the set
of the usual smooth maps $\matR\to\matR^n$ (thus, its plots are
characterized as those maps that locally factor through the smooth
maps $\matR\to\matR^n$). For $n\geqslant 2$ this diffeology is
different from the standard one (see \cite{iglesiasBook}, Sect.
1.10), although the underlying D-topology is the same (see
\cite{CSW_Dtopology}).

Of course, when we want to carry this notion over to one of our
regular trees, the first question to consider is, \emph{which maps
take the place of smooth ones?} We speak about this in detail later
on, but in brief, the main points are: the set of all maps $\matR\to
T$ would produce a very, and perhaps unreasonably, large diffeology,
the set of all continuous maps still gives a very large one (see
below for the curious observation of how the Peano curve enters the
picture in this respect), and so it seems reasonable to settle for
the set of all embeddings $\matR\hookrightarrow T$ as the generating
set for the wire diffeology on $T$.

\subsection{The wire diffeology on $T$}

As has already been mentioned, such diffeology is the one generated
by some subset of the set of all maps $\{\matR\to T\}$; the question
is, which subset? The following easy considerations suggest to
discard the ``extreme'' possibilities, more specifically: the
coarsest of such diffeologies is the one consisting of \emph{all}
maps $\matR\to T$, whereas the finest one is the discrete
diffeology, \emph{i.e.} the one generated by all constant maps
$\matR\to T$. Neither of the two is very interesting (as is
generally the case), and neither respects the structure of $T$ as a
topological space, something that we do want to take into account.

This latter consideration suggests to consider continuous maps only,
and our options become, to take \emph{all} continuous maps or only
some of them (such as, for instance, the injective ones, which is
what we will end up doing). We now illustrate that the diffeology
generated by the set of all continuous maps $\matR\to T$ (which for
the moment we will call the \emph{coarse wire diffeology}) is still
very large and, in some very informal sense, loses the
$1$-dimensional nature of $T$.

\paragraph{The coarse wire diffeology and the Peano curve} The above
statement that the just-mentioned coarse wire diffeology does not
truly respect the $1$-dimensional nature of our trees, can actually
be observed immediately from the famous example of the \emph{Peano
curve}, a continuous curve that fills the entire unit square.
Furthermore, after the appearance in 1890 of the ground-breaking
Peano's example, it became known that \emph{any} $\matR^n$ (with $n$
an arbitrary positive integer number) is the range of some
continuous curve; to be precise, for any $n=2,3,\ldots$ there exists
a continuous surjective map $s_n:\matR\to\matR^n$ (hence onto any
domain of $\matR^n$). Although none of these maps is invertible,
they do allow for a sort of immersion of \emph{any other} diffeology
into the coarsest wire diffeology, by assigning to a given plot
$p:\matR^n\supseteq U\to T$ the composition $p\circ t_U \circ s_n$
(where $t_U$ is some diffeomorphism $\matR^n\to U$, fixed for each
$U$). Although this assignment would not be one-to-one, it does give
an (intuitive, if nothing else) idea of how large the coarse wire
diffeology is.

\paragraph{The embedded wire diffeology on $T$} This is the
diffeology that we settle one; it is the diffeology generated by all
injective and continuous in both directions maps $\matR\to T$. It
depends on $T$ only, so we denote it by $\calD_T$. We furthermore
denote the generating set of $\calD_T$, the set $\{f:\matR\to T\,|\,
f\mbox{ is injective and both ways continuous}\}$, by
$\tilde{\calP}_T$.

The first thing that we would like to do is to restrict this
generating set as much as possible; indeed, if two maps,
$f_1,f_2:\matR\to T$, are such that $f_2=f_1\circ g$ for some
diffeomorphism $g:\matR\to\matR$ then (as it follows from the
definition of a generated diffeology) only one of them needs to
belong to the generating set. Therefore we denote by $\calP_T$ the
quotient of $\tilde{\calP}_T$ by the (right) action of the group of
diffeomorphisms of $\matR$; when it does not create confusion, by
one or more elements of $\calP_T$ we will mean a corresponding
collections of maps that are specific representatives of some
equivalence classes. The above observations then prove the
following:

\begin{claim}
The diffeology $\calD_T$ is generated by $\calP_T$.
\end{claim}

\paragraph{The topology} We now proceed to showing that the diffeology
chosen does satisfy the condition that we wanted to, namely, that
the following is true.

\begin{thm}
The D-topology corresponding to $\calD_T$ is the usual topology of
$T$.
\end{thm}

\begin{proof}
Recall that by the definition of D-topology a set $X'\subset T$ is
D-open if and only if for any plot $p:U\to T$ the pre-image
$p^{-1}(X)$ is open in $U$; now, by construction and by the
definition of the generated diffeology this is equivalent to
$\gamma^{-1}(X)$ being open in $\matR$ for any $\gamma\in\calP_T$.

We need to show that $X'$ is D-open if and only if it is open in $T$
in the usual sense. Suppose first that $X'$ is open. Then its
pre-image with respect to any $\gamma$ is open in $\matR$ because
$\gamma$ is continuous; therefore it is D-open by the very
definition of D-openness.

Now suppose that $X'$ is a D-open set; we need to show that it is
also open in the usual sense. To do so, it is sufficient to show
that for any point of $X'$ the latter contains its open
neighbourhood. Choose such an arbitrary point $x\in X'$; we consider
two cases.

Suppose first that $x$ belongs to the interior of some edge $e$. Set
$X=X'\cap\mbox{Int}(e)$, and let $\gamma\in\calP_T$; we can assume
that its image contains $e$. Note that since $\gamma$ is injective,
we have
$\gamma^{-1}(X)=\gamma^{-1}(X')\cap\gamma^{-1}(\mbox{Int}(e))$; both
of these two sets are open in $\matR$, the first because $X'$ is
D-open and the second because it is the pre-image of an open set
under the continuous map $\gamma$. This implies that
$\gamma^{-1}(X)$ is open in $\matR$, therefore $X$ is open in the
image of $\gamma$, the latter being a homeomorphism with its image,
and it is open in $\mbox{Int}(e)$, hence it is open in $T$ as well.
Thus, $X$ is an open neighbourhood of $x$ contained in $X'$.

Suppose now that $x$ is a vertex (we can assume that it is not the
root; the proof changes only formally for the latter). Let
$e_1,\ldots,e_{p+1}$ be the edges incident to $x$. For each
$1\leqslant i<j\leqslant p+1$ let
$X_{i,j}=\mbox{Int}(e_i)\cup\mbox{Int}(e_j)\cup\{x\}$; set
$$X=\cup_{i,j}(X'\cap X_{i,j}).$$

We need to show that $X$ is open in $T$. For each $i,j$ choose a map
$\gamma_{i,j}$ such that its image contains $e_i\cup e_j$. Then
$\gamma_{i,j}^{-1}(X'\cap
X_{i,j})=\gamma_{i,j}^{-1}(X')\cap\gamma_{i,j}^{-1}(X_{i,j})$; both
of these sets are open in $\matR$, by the D-openness of $X'$ and by
continuity of $\gamma$. Hence $\gamma_{i,j}^{-1}(X'\cap X_{i,j})$ is
open in $\matR$ and, $\gamma_{i,j}$ being a homeomorphism with its
image, the set $X'\cap X_{i,j}$ is open in $e_i\cup e_j$. It follows
that $X$ is open in $e_1\cup\ldots\cup e_{p+1}$ and therefore it is
open in $T$; thus, it is an open nieghbourhood of $x$ contained in
$X'$, and this concludes the proof.
\end{proof}

\section{$\Aut T$ as a diffeological group}

In this section we consider $T$ endowed with the embedded wire
diffeology described in the previous section. We must first ensure
that the elements of $\Aut T$ are smooth maps with respect to this
diffeology; this then gives rise to the functional diffeology on
$\Aut T$ and to the \emph{a priori} finer diffeology that makes
$\Aut T$ into a diffeological group and is the finest one with such
property.

\subsection{The functional diffeology on $\Aut T$}

In this section we first make some observations regarding the plots
of the functional diffeology on $\Aut T$; as a preliminary, we need
to show that such diffeology is indeed well-defined, \emph{i.e.},
that the elements of $\Aut T$ are indeed diffeomorphisms. We then
proceed to consider the D-topology underlying the functional
diffeology of $\Aut T$.

\paragraph{Automorphisms as diffeomorphisms} The following statement
follows easily from the choice of diffeology on $T$.

\begin{prop}
Let $g\in\Aut T$. Then $g:T\to T$ is a smooth map with respect to
the diffeology $\calD_T$.
\end{prop}

\begin{proof}
By definition of a generated diffeology and that of a smooth map it
is sufficient to show that for any given injective and both ways
continuous map $\gamma:\matR\to T$ the composition $g\circ\gamma$ is
again injective and both ways continuous. This follows from the fact
that $g$ is an automorphism of $T$, \emph{i.e.}, it is a
homeomorphism of $T$ considered with its usual topology; as we have
already established that the D-topology of $T$ coincides with the
usual one, this proves the claim.
\end{proof}

\paragraph{A special family of plots of $T$} In the arguments that follow,
we will make use of the following family of plots of $T$. Let
$\hat{\gamma}$ be an infinite path in $T$; let $v_0\in\hat{\gamma}$
be the vertex of the smallest length. For each $\hat{\gamma}$ we fix
a homeomorphism $\gamma:\matR\to\hat{\gamma}\subset T$ such that
$$\gamma(0)=v_0,\mbox{and for any }n\in\matZ\,\,\,\,\gamma(n)\mbox{ is a vertex of length }|v_0|+n.$$
Obviously, every $\gamma$ is a plot for the diffeology $\calD_T$. We
denote the set of maps $\gamma$, associated to all possible
$\hat{\gamma}\subset T$, by
$$\Gamma(T)=\{\gamma\,|\;\hat{\gamma}\mbox{ an infinite path in }T\}.$$
We denote by $\Gamma_0(T)$ the subset of $\Gamma(T)$ consisting of
all those maps whose image contains the root.

For technical reasons we wish to stress that all maps
$\gamma\in\Gamma(T)$ possess, by construction, the following
properties:
\begin{itemize}
  \item for any given $x\in\matR$, its image $\gamma(x)$ is a vertex
  if and only if $x\in\matZ$;
  \item in particular, the restriction of $\gamma$ on any interval of form
  $(n,n+1)$ is a homeomorphism with the interior of some edge of
  $T$.
\end{itemize}

\paragraph{Smooth curves in $\Aut T$} Since the D-topology is
defined by smooth curves (as mentioned in the first section, see
\cite{CSW_Dtopology}), we first establish the following
characterization of those plots of the functional diffeology on
$\Aut T$ that are curves.

\begin{prop}
Let $p:\matR\to\Aut T$ be a plot for the functional diffeology on
$\Aut T$. Then for all $m,n\in\matN$ the automorphisms $p(n)$,
$p(n+1)$ belong to the same coset of $\stab(n)$, and the
automorphisms $p(-m)$, $p(-m-1)$ belong to the same coset of
$\stab(m)$.
\end{prop}

\begin{proof}
Recall that by Proposition \ref{criterio-funct-diff} $p$ is a plot
if and only if the map $\varphi_p:\matR\times T\to T$ given by
$\varphi_p(x,v)=p(x)(v)$ is smooth. The latter condition implies, in
particular, that for any smooth map $f:\matR\to\matR$ and for any
injective two ways continuous map $\gamma:\matR\to T$ the
composition $\varphi_p\circ(f,\gamma):\matR\to T$ is a plot of $T$,
\emph{i.e.}, that (at least locally) it is the composition
$\tilde{\gamma}\circ\tilde{f}$, of some smooth map
$\tilde{f}:\matR\to\matR$ and some injective two ways continuous
$\tilde{\gamma}:\matR\to T$. In particular, the map
$\varphi_p\circ(f,\gamma)$ is a continuous map in the usual sense.

Let us now fix a positive integer $n$, a vertex $v$ of length $n$,
and a vertex $v'$ of length $n+1$ adjacent to $v$. Let
$\gamma\in\Gamma_0(T)$ be  such that $\gamma(n)=v$ and
$\gamma(n+1)=v'$. By definition of $\varphi_p$ we have that
$$(\varphi_p\circ(\mbox{Id},\gamma))(n)=\varphi_p(n,\gamma(n))=p(n)(v),\mbox{ and}$$
$$(\varphi_p\circ(\mbox{Id},\gamma))(n+1)=\varphi_p(n+1,\gamma(n+1))=p(n+1)(v').$$

We claim that $p(n)(v)$ and $p(n+1)(v')$ are adjacent vertices. That
they are vertices, of which the first is has length $n$ and the
second one has length $n+1$, is obvious, since $p$ takes values in
$\Aut T$, all of whose elements send vertices to vertices preserving
their length. It suffices to show that they are adjacent,
\emph{i.e.}, joined by an edge. As we have already observed, the map
$\varphi_p\circ(\mbox{Id},\gamma)$ is continuous in the usual sense,
so it suffices to show that the image of the interval $(n,n+1)$
under it does not contain vertices. Indeed, by its definition
$\varphi_p\circ(\mbox{Id},\gamma)$ writes as
$(\varphi_p\circ(\mbox{Id},\gamma))(x)=p(x)(\gamma(x))$; we first
observe that this image is a vertex if and only if $\gamma(x)$ is a
vertex (this is because $p(x)\in\Aut T$), then, second, $\gamma(x)$
is a vertex if and only if $x\in\matZ$ (this is by choice of
$\gamma$). In particular, if $n<x<n+1$ then
$(\varphi_p\circ(\mbox{Id},\gamma))(x)=p(x)(\gamma(x))$ belongs to
the interior of some edge, and precisely, the edge that joins
$p(n)(v)$ and $p(n+1)(v')$.

It remains to observe that $p(n+1)(v')$ is adjacent to a unique
vertex of length $n$; since $v$ and $v'$ are adjacent, $v$ has
length $n$, and $p(n+1)$ is an automorphism, this vertex is
$p(n+1)(v)$. On the other hand, $p(n)(v)$ has length $n$, and we
have just shown that it is adjacent to $p(n+1)(v')$; we conclude
that
$$p(n)(v)=p(n+1)(v).$$
Finally, since $v$ is arbitrary, we can conclude that
$p(n)\stab(n)=p(n+1)\stab(n)$, as claimed; and since $n$ is
arbitrary, this proves the entire statement.
\end{proof}

We now can draw the following conclusion.

\begin{cor}
Each of the two sequences $\{p(n)\}$, $\{p(-m)\}$ is a converging
sequence for the congruence topology on $\Aut T$.
\end{cor}

Note that we phrase this statement in terms of the congruence
topology on $\Aut T$, and not in those of the profinite topology,
which for $\Aut T$ does coincide with the congruence one. This is to
highlight the relation of this statement for examples such as the
group $\Gamma$, for which the two topologies are different; although
we will see shortly a fact that renders the difference
insignificant.

\begin{proof}
It suffices to show that $p(n)\stab(n)=p(n+k)\stab(n)$ for all
$n,k\in\matN$; as in the previous proof, this is equivalent to
having $p(n)(v)=p(n+k)(v)$ for any vertex $v$ of length $n$. Choose
such a vertex, and fix a map $\gamma\in\Gamma_0(T)$ such that
$\gamma(n)=v$. As we have established in the proof of the previous
proposition, $p(n+1)(\gamma(n+1))$ belongs to the subtree
$T_{p(n)(v)}$. By the same reasoning, applied to $n+1$, the vertex
$p(n+2)(\gamma(n+2))$ belongs to the subtree
$T_{p(n+1)(\gamma(n+1))}\subset T_{p(n)(v)}$. Now, since
$\gamma(n+2)\in T_v$ by choice of $\gamma$, each vertex of length
$n+2$ belongs to a unique subtree $T_u$ with $|u|=n$, and $p(n+2)$
is an automorphism, we can conclude that $p(n+2)(v)=p(n)(v)$, and
the corollary follows by the now obvious induction on $k$.
\end{proof}

The above Corollary actually paves the way to the following
statement of crucial consequences:

\begin{prop}\label{D-top:discrete:prop}
Let $p:\matR\to\Aut T$ be a plot for the functional diffeology. Then
$p$ is a constant map.
\end{prop}

\begin{proof}
Recall that by Proposition \ref{criterio-funct-diff} $p$ is a plot
if and only if the map $\varphi:\matR\times T\to T$ given by
$\varphi(u,t)=p(u)(t)$ is smooth. Now, by definition $\varphi$ is
smooth if and only if for any plot $p':U\to\matR\times T$ the
composition $\varphi\circ p'$ is again a plot of $T$.

Let us fix and arbitrary vertex $v$ of $T$, and let us take, as the
plot $p'$, the following map:
$(\mbox{Id}_{\matR},c_v):\matR\to\matR\times T$, where $c_v:\matR\to
T$ is the constant map acting by $c_v(x)\equiv v$. Then
$$(\varphi\circ(\mbox{Id}_{\matR},c_v))(x)=p(x)(v).$$
Observe now that $p(x)$ is an automorphism of $T$ for all $x$, and
so sends vertices to vertices; therefore the image of the map
$\varphi\circ(\mbox{Id}_{\matR},c_v)$ is a set of vertices of $T$.
In particular, it is a discrete subset of $T$.

On the other hand, $\varphi\circ(\mbox{Id}_{\matR},c_v)$ is a plot
of $T$; as such, it is either a constant map or it filters through
an injective continuous map $\matR\to T$ via a smooth map
$\matR\to\matR$. In this latter case it must a continuous map
defined on a connected set and so cannot have a discrete set with
more than one point as its image. It remains to conclude that
$\varphi\circ(\mbox{Id}_{\matR},c_v)$ is a constant map, which means
that $p(x)(v)$ does not depend on $x$. Since $v$ is arbitrary, this
implies that $p$ is a constant map, as is claimed.
\end{proof}

The meaning of this Proposition is that the only smooth curves in
$\Aut T$ are the constant ones; this has far-reaching consequences,
as we immediately see.

\paragraph{The D-topology of $\Aut T$} From what is established in
the previous paragraph we easily draw the following conclusion.

\begin{thm}
The D-topology underlying the functional diffeology on $\Aut T$ is
the discrete topology.
\end{thm}

\begin{proof}
This follows from the Proposition above and Theorem 3.7 of
\cite{CSW_Dtopology} (see also Example 3.2(2) therein).
\end{proof}

\paragraph{The functional diffeology of $\Aut T$ is discrete}
Moreover, we consider the proof of Proposition
\ref{D-top:discrete:prop}, we see that the plot $p$ under
consideration being defined on $\matR$ rather than an arbitrary
domain $U\subseteq\matR^k$ was not significant; it would hold just
the same writing $U$ in place of $\matR$ and $u\in U$ in place of
$x$. This implies that \emph{all} plots of the functional diffeology
of $\Aut T$ are constant maps, and therefore this diffeology is
indeed discrete.

\subsection{Functional diffeology of $\Aut T$ and its diffeological
group structure}

We shall now make some remarks regarding the diffeological group
structure on $\Aut T$, in relation to its functional diffeology. We
have already established that the latter is discrete and therefore
is the finest possible diffeology on $\Aut T$ (see
\cite{iglesiasBook}, Section 1.20). For this reason it coincides
with the diffeological group structure, the latter being \emph{a
priori} finer than the functional diffeology.

\vspace{1cm}

\noindent University of Pisa \\
Department of Mathematics \\
Via F. Buonarroti 1C\\
56127 PISA -- Italy\\
\ \\
ekaterina.pervova@unipi.it\\

\end{document}